\documentclass[12pt,twoside]{amsart}
\usepackage{amsfonts, amsmath, amsopn, amssymb, latexsym, graphics, graphicx, enumerate,hyperref}
\allowdisplaybreaks

\usepackage{tikz}
\usetikzlibrary{matrix,arrows,decorations.pathmorphing}
\usepackage{pgfplots}

\DeclareFontFamily{OT1}{rsfs}{}
\DeclareFontShape{OT1}{rsfs}{n}{it}{<-> rsfs10}{}
\DeclareMathAlphabet{\mathscr}{OT1}{rsfs}{n}{it}

\usepackage{arydshln}
\makeatletter
  \renewcommand*\env@matrix[1][*\c@MaxMatrixCols c]{%
    \hskip -\arraycolsep
    \let\@ifnextchar\new@ifnextchar
  \array{#1}}
\makeatother

\newcommand{\Q}{\mathbb{Q}}
\newcommand{\R}{\mathbb{R}}
\newcommand{\C}{\mathbb{C}}

\numberwithin{equation}{section}

\newcommand{\specialcell}[2][c]{\begin{tabular}[#1]{@{}l@{}}#2\end{tabular}}

\begin{document}

\newtheorem{thm}[subsection]{Theorem}
\newtheorem*{thm*}{Theorem}
\newtheorem{lem}[subsection]{Lemma}
\newtheorem{cor}[subsection]{Corollary}
\newtheorem*{cor*}{Corollary}
\newtheorem{prop}[subsection]{Proposition}
\newtheorem{ques}{Question}

\newtheorem{mainthm}{Theorem}
\renewcommand*{\themainthm}{\Alph{mainthm}}

\theoremstyle{definition}
\newtheorem{none}[subsection]{}
\newtheorem{lect}[subsection]{Lecture}
\newtheorem{rem}[subsection]{Remark}
\newtheorem{ex}[subsection]{Example}
\newtheorem{Def}[subsection]{Definition}
\newtheorem{exer}[subsection]{Exercise}
\newtheorem{conj}[subsection]{Conjecture}
\newtheorem{prob}[subsection]{Problem}
\newtheorem{con}[subsection]{Construction}

\title[Totally Geodesic Spectra of Quaternionic Hyperbolic Orbifolds]{Totally Geodesic Spectra\\ of\\ Quaternionic Hyperbolic Orbifolds}
\author{Jeffrey S. Meyer}
\address{Department of Mathematics\\ 
University of Oklahoma\\ 
Norman, OK 73019 USA}
\email[]{jmeyer@math.ou.edu}

\maketitle
\pagestyle{myheadings}


\let\thefootnote\relax\footnote{\textit{Date: \today} \hfill \tt{jmeyer@math.ou.edu}}

\vspace{-2pc}


\begin{abstract}
In this paper we analyze and classify the totally geodesic subspaces of finite volume quaternionic hyperbolic orbifolds and their generalizations, locally symmetric orbifolds arising from irreducible lattices in Lie groups of the form $(\mathbf{Sp}_{2n}(\R))^q \times \prod_{i=1}^r \mathbf{Sp}(p_i,n-p_i) \times (\mathbf{Sp}_{2n}(\C))^s$.
We give criteria for when the totally geodesic subspaces of such an orbifold determine its commensurability class.
We give a parametrization of the commensurability classes of finite volume quaternionic hyperbolic orbifolds in terms of arithmetic data, 
which we use to show that the complex hyperbolic totally geodesic subspaces of a quaternionic hyperbolic orbifold  determine its commensurability class, but the real hyperbolic totally geodesic subspaces do not.
Lastly, our tools allow us to show that every cocompact lattice $\Gamma<\mathbf{Sp}(m,1)$, $m\ge 2$, 
contains quasiconvex surface subgroups.  
\end{abstract}


\section{Introduction}

\qquad 
Quaternionic hyperbolic space $\mathbf{H}_{\mathbb{H}}^m$ is one of the four classes of rank one Riemannian globally symmetric spaces of noncompact type \cite[\S 19]{Mostow}.
A finite volume quaternionic hyperbolic orbifold is a locally symmetric space arising as a quotient of $\mathbf{H}_{\mathbb{H}}^m$ by a lattice $\Gamma\subset \mathrm{Isom}(\mathbf{H}_{\mathbb{H}}^m)\cong \mathbf{PSp}(m,1)$.   
Such spaces are quaternion-K\"ahler manifolds with negative sectional curvature and, furthermore, are Einstein manifolds.   
In this paper, we analyze and parametrize the totally geodesic subspaces of finite-volume quaternionic hyperbolic orbifolds and their generalizations. 
In particular, we are guided by the following question.

\begin{ques}\label{ques:main}
Let $\mathcal{X}$ denote a set of finite volume Riemannian orbifolds.  Suppose  $M, M'\in \mathcal{X}$ have the property that each finite volume nonflat totally geodesic subspace
$N\subset M$ is commensurable to a totally geodesic subspace $N'\subset M'$.
Must $M$ and $M'$ be commensurable?
\end{ques}

\qquad 
There are obvious refinements of this question where one restricts the collection of totally geodesic subspaces being considered to ones with some additional property $\mathcal{P}$.
The \textbf{rational totally geodesic spectrum relative to} $\mathcal{P}$ of a Riemannian orbifold $M$ is the set
\begin{align}\label{defqtgm}
\Q TG(M, \mathcal{P}):=\left\{ \parbox{3.8in}{\begin{center}Commensurability classes of nonflat, finite volume, totally geodesic subspaces of  $M$ satisfying property $\mathcal{P}$\end{center}}\right\}.
\end{align}
Let $\Q TG(M):=\Q TG(M,\varnothing)$.
We write $M\sim_cM'$ when $M$ and $M'$ are commensurable.  
Question 1 can be succinctly restated as: \textit{For $M,M'\in \mathcal{X}$, does $\Q TG(M)=\Q TG(M')$ imply $M\sim_cM'$?}
The analysis of this question begins by identifying an interesting yet tractable class $\mathcal{X}$ such as a collection of arithmetic Riemannian locally symmetric orbifolds of a fixed Killing--Cartan type.  
In such a case, one may then use the classification of simple algebraic groups over number fields to analyze these spaces.


\qquad 
There are several classes of spaces for which there is a positive answer to Question \ref{ques:main}, 
including arithmetic real hyperbolic 3-manifolds that contain a totally geodesic surface \cite{McR}, standard arithmetic real hyperbolic $m$-manifolds, $m\ge4$ \cite{M14spec}, and more generally
arithmetic locally symmetric orbifolds of type $B_n$ or $D_n$ coming from quadratic forms of dimension at least $5$ where either both spaces are simple or have isomorphic fields of definition. \cite{M14spec}.
However, there are classes of spaces for which there are negative answers as well.  
There are infinitely many commensurability classes of arithmetic real hyperbolic 3-manifolds with no totally geodesic surfaces \cite[Thm 9.5.1]{MaR2}.  
More generally, there are noncommensurable arithmetic manifolds arising from semisimple Lie groups of the form $(\mathbf{SL}_n(\R))^r \times(\mathbf{SL}_n(\C))^s$ (i.e., of inner type $A_{n-1}$) that have the same commensurability classes of totally geodesic surfaces coming from a fixed field \cite{McRey}.

\qquad 
In this paper, we address Question \ref{ques:main} for the class of spaces coming from absolutely almost simple $k$-groups of Killing--Cartan type $C_n$, where $k$ is a number field and $n\ge 3$.  
These spaces are locally symmetric orbifolds arising from irreducible lattices in semisimple Lie groups of the form $(\mathbf{Sp}_{2n}(\R))^q \times \prod_{i=1}^r \mathbf{Sp}(p_i,n-p_i) \times (\mathbf{Sp}_{2n}(\C))^s$, $n\ge 3$.  
In Section \ref{section:construction} we give an explicit construction of these lattices.  
In particular, finite volume quaternionic hyperbolic orbifolds are examples of such spaces.
We use this construction to give a parametrization of commensurability classes of finite volume quaternionic hyperbolic $4m$-orbifolds, $m\ge2$, in terms of \textbf{admissible triples} $(k,v_0,D)$ where $k$ is a totally real number field, $v_0$ is a distinguished real place of $k$, and $D$ is a quaternion division algebra over $k$ that ramifies at every infinite place of $k$ (Definition \ref{admissibledef}).

\begin{mainthm}
[Parametrization of commensurability classes of finite volume quaternionic hyperbolic orbifolds.]
\label{thrmC}
For each $m\ge2$, there is a bijection between commensurability classes of finite volume quaternionic hyperbolic $4m$-orbifolds and equivalence classes of \hyperref[admissibledef]{admissible triples} $(k,v_0,D)$.  
Furthermore, for each such commensurability class, $k$ is its field of definition and $D$ is its algebra of definition.
\end{mainthm}

\qquad 
This parametrization may be known to some experts, but it has not appeared in the literature.  
Theorem \ref{thrmC} is the quaternionic hyperbolic analogue of of Maclachlan's parametrization of commensurability classes of standard arithmetic real hyperbolic spaces \cite{Mac} and McReynold's parametrization of commensurability classes of  arithmetic complex hyperbolic spaces of  first type \cite[Chapter 5]{McRey2}.

\qquad 
We say a Lie group is \textbf{simple} if the complexification of its Lie algebra is simple (which is to say the algebraic $\R$-group corresponding to its adjoint group is absolutely simple).
In particular, in the case of groups of type $C_n$, it means $q+r=1$ and $s=0$.
Throughout this paper, $M$ will denote a locally symmetric space and $\widetilde{M}$ its globally symmetric universal cover.
A locally symmetric space $M$ is \textbf{simple} if $\mathrm{Isom}(\widetilde{M})$ is a simple Lie group, and we denote its field of definition  by $k(M)$ \cite[Section 7]{M14spec}.

\begin{mainthm}
\label{mainthm:dimfield}
Let $M_1$ and $M_2$ be simple finite volume locally symmetric orbifolds of type $C_{n_1}$ and $C_{n_2}$, respectively, where $n_i\ge 3$, and let $\overline{G}_i=\mathrm{Isom}(\widetilde{M}_i)$. 
If $\Q TG(M_1)=\Q TG(M_2)$, then
\begin{enumerate}[\qquad(a)]
\item $\overline{G}_1$ and $\overline{G}_2$ are isomorphic as Lie groups.  In particular, 
\begin{enumerate}[\ (i)]
\item $n_1=n_2=:n$,
\item $\mathrm{rank}\, M_1 = \mathrm{rank}\,  M_2 $,
\item $\dim M_1 = \dim M_2 $.
\end{enumerate}
\item Furthermore, if $n\ge 4$, then $k(M_1)$ and $k(M_2)$ are isomorphic.
\end{enumerate}
\end{mainthm}

\qquad 
Using the theory of Hermitian forms over division algebras over number fields \cite{Sch} and the classification of algebraic groups \cite{TitsClassification}, we are then able to prove the following theorem about locally symmetric spaces of type $C_n$.

\begin{mainthm}
\label{mainthm:commensurable}
Let $M_1$ and $M_2$ be finite volume locally symmetric spaces of type $C_{n_1}$ and $C_{n_2}$ respectively, where $n_i\ge 3$,
such that $k(M_1)$ and $k(M_2)$ are isomorphic.  
If $\Q TG(M_1)=\Q TG(M_2)$, then $M_1\sim_cM_2$.
\end{mainthm}

\qquad 
In particular, Theorem \ref{mainthm:dimfield} and Theorem \ref{mainthm:commensurable} together imply (Corollary \ref{cor:bandc}) that the commensurability class of a simple finite volume locally symmetric space of type $C_n$, $n\ge 4$ is completely determined by its finite volume totally geodesic subspaces.

\qquad 
A quaternionic hyperbolic $4m$-orbifold $M$ has three types of proper totally geodesic subspaces:  quaternionic hyperbolic,  complex hyperbolic, and real hyperbolic (see Theorem \ref{totgeoclassification}).  Such a totally geodesic subspace is \textbf{maximal} if it attains dimension: $4m-4$, $2m$ and $m$, respectively.   We analyze the role that each type plays in determining the commensurability class of $M$.

\qquad We let $\mathcal{X}_{\mathbb{H}}$ denote the set of finite volume quaternionic hyperbolic orbifolds of dimension at least $8$ and let  $\mathcal{P}_{\mathbb{H}}$ (resp. $\mathcal{P}_{\mathbb{C}}$, $\mathcal{P}_{\mathbb{R}}$) be the property of being maximal quaternionic (resp.  complex,  real) hyperbolic.
(For example, in this notation, $\Q TG(M, \mathcal{P}_{\mathbb{C}})$ is the set of commensurability classes of maximal, finite volume, complex hyperbolic, totally geodesic subspaces of $M$.)

\begin{mainthm}
\label{mainthm:complex}
If $M_1,M_2\in \mathcal{X}_{\mathbb{H}}$ and either $\Q TG(M_1, \mathcal{P}_{\mathbb{C}})=\Q TG(M_2, \mathcal{P}_{\mathbb{C}})$ or $\Q TG(M_1, \mathcal{P}_{\mathbb{H}})=\Q TG(M_2, \mathcal{P}_{\mathbb{H}})$, then $M_1\sim _cM_2$.
\end{mainthm}

\qquad 
In other words, Theorem \ref{mainthm:complex} says that the complex hyperbolic (resp. quaternionic hyperbolic) totally geodesic subspaces determine a commensurability class.  
In contrast, real hyperbolic subspaces fail to determine a commensurability class.

\begin{mainthm}
\label{mainthm:real}
For every $s\in \mathbb{Z}_{>1}$, there exists a family of pairwise noncommensurable spaces $M_1, M_2,\ldots, M_s\in \mathcal{X}_{\mathbb{H}}$ such that $\Q TG(M_i, \mathcal{P}_{\mathbb{R}})=\Q TG(M_j, \mathcal{P}_{\mathbb{R}})$ for all $1\le i,j\le s$.
\end{mainthm}

\qquad 
Theorems \ref{mainthm:complex} and \ref{mainthm:real} can be thought of as an orbifold analogue of fact that a quaternion division algebra over a number field is determined by its collection of maximal subfields, but not by its center.

\qquad 
A \textbf{surface group} is the fundamental group of a closed surface (that is not the sphere).
Understanding if and how surface groups sit inside certain word-hyperbolic groups has been a major line of inquiry in recent years, and even played a prominent role in resolving Thurston's virtual Haken conjecture \cite{KahnMarkovic} \cite{BergeronWise} \cite{Agol}.
Gromov has conjectured that every one-ended word-hyperbolic group contains a surface subgroup.
In \cite{Hamenstadt}, it was shown that cocompact lattices in rank one simple Lie groups of noncompact type distinct from $\mathbf{SO}(2m,1)$, $m\ge1$, contain surface subgroups.
A consequence of the results of \cite{M14spec} is that arithmetic cocompact lattices in groups of type $\mathbf{SO}(2m,1)$, $m\ge2$, contain quasiconvex surface subgroups.
We use our techniques to prove the following.

\begin{mainthm}
\label{mainthm:surfacesubgroup}
If $\Gamma<\mathbf{Sp}(m,1)$, $m\ge 2$, is a cocompact lattice, then $\Gamma$ contains quasiconvex surface subgroups.
\end{mainthm}


\qquad 
In fact, in Theorem \ref{thm:relhyp} we show that all nonuniform lattices in $\mathbf{Sp}(m,1)$ contain surface subgroups as well.  It should also be noted that the methods used in these theorems are arithmetic in nature, and hence very different from the methods of \cite{Hamenstadt}.

\phantomsection
\label{sectionack}
\vspace{0.5pc}
\noindent\textbf{Acknowledgments.}
We would like to thank Matthew Stover for initially bringing this problem to our attention and giving valuable feedback on the early versions of this paper.  
We also would like to thank Lucy Lifschitz, Benjamin Linowitz,  and D.B. McReynolds for many useful and interesting discussions.
\vspace{0.5pc}

\section{Quaternionic Hyperbolic Geometry}

\qquad 
As compared to real hyperbolic space, $\mathbf{H}_{\mathbb{R}}^m$, and complex hyperbolic space, $\mathbf{H}_{\mathbb{C}}^m$, there is relatively little in the literature about $4m$-dimensional quaternionic hyperbolic space $\mathbf{H}_{\mathbb{H}}^m$ (\cite[\S19]{Mostow}, \cite{KPar}, \cite{KPan}).
For the reader unfamiliar with $\mathbf{H}_{\mathbb{H}}^m$ we devote this section to describing its geometry.
In particular, in Theorem \ref{totgeoclassification} we give a classification of the totally geodesic subspaces of $\mathbf{H}_{\mathbb{H}}^m$ that will be important in later sections.
It should be noted that, while the theory in many ways is similar to that of complex hyperbolic spaces, much greater care must be taken in definitions and computations due to the noncommutativity of the quaternions.

\subsection{Notation}
Throughout this section, we shall use the following notation.  Let
\begin{itemize}
\item $\mathbb{H}=\left(\frac{-1,-1}{\R}\right)$  be Hamilton's quaternions over $\R$.
\item $\overline{\alpha}$ be the conjugate of $\alpha\in \mathbb{H}$.
\item $V=\mathbb{H}^{m+1}$ denote the $m+1$-dimensional $\mathbb{H}$-vector space with both left and right $\mathbb{H}$-action.
\item $h$  denote the canonical Hermitian form on $V$ with signature $(m,1)$. i.e., for $\mathbf{v}, \mathbf{w}\in V$, $$h(\mathbf{v},\mathbf{w})=\sum_{i=1}^{m}\overline{v_i}w_i - \overline{v_{m+1}}w_{m+1}$$
where $\mathbf{v}=(v_1,v_2,\ldots, v_{m+1})$ and $\mathbf{w}=(w_1,w_2,\ldots, w_{m+1})$.
\item $\mathbb{H}^{m,1}$ denote the Hermitian pair $(V, h)$.
\item $[\mathbf{v}]$ denote the set $\{\mathbf{v}\alpha\}$ where $\mathbf{v}\in V$ is fixed and $\alpha$ ranges over $\mathbb{H}$.  We call $[\mathbf{v}]$ an $\mathbb{H}$-line.
\item $\mathbb{P}(V)=\{[\mathbf{v}]\ | \ \mathbf{v}\in V\}$.
\item $\mathbf{v}\in V$ (resp. $[\mathbf{v}]\in \mathbb{P}(V)$) be called a \textbf{negative vector} (resp. \textbf{negative line}) if $h(\mathbf{v},\mathbf{v})<0$.
\item $\mathbf{GL}_{m+1}(\mathbb{H})$ be the set of invertible $(m+1)\times (m+1)$ matrices with entries in $\mathbb{H}$.  This can naturally be identified with invertible right-$\mathbb{H}$-linear maps of $V$, $T(\mathbf{v}\alpha)=T(\mathbf{v})\alpha$.
\item $\mathbf{Sp}(m,1)=\{A\in \mathbf{GL}_{m+1}(\mathbb{H})\ | \ h(A\mathbf{v},A\mathbf{w})=h(\mathbf{v},\mathbf{w}) \mbox{ for all } \mathbf{v},\mathbf{w}\in \mathbb{H}^{m,1}\}$.
If we let $$H=\begin{pmatrix} I_{m\times m} & 0 \\ 0 & -1\end{pmatrix},$$ then this amounts to  
$$\mathbf{Sp}(m,1)=\{A\in \mathbf{GL}_{m+1}(\mathbb{H})\ | \ {}^T\overline{A}HA=H\}.$$
\item $\mathfrak{sp}(m,1)=\{A\in \mathfrak{gl}_{m+1}(\mathbb{H})\ | \ H({}^T\overline{A})H+A=0\}$.\\
\item $\mathfrak{sp}(m,1)=\mathfrak{k}\oplus \mathfrak{p}$ is the standard Cartan decomposition of $\mathfrak{sp}(m,1)$ where
$$\mathfrak{k}:=\left\{ \begin{bmatrix} X& \mathbf{0}\\ {\mathbf{0}} & Y\end{bmatrix} \  \bigg| \  X\in \mathfrak{sp}(m), Y\in \mathfrak{sp}(1)\right\}
\qquad \mbox{and} \qquad
\mathfrak{p}:=\left\{ \begin{bmatrix} 0_{m\times m}& \mathbf{v}\\ {}^T\overline{\mathbf{v}} & 0\end{bmatrix} \  \bigg| \  \mathbf{v}\in \mathbb{H}^m\right\}.
$$
\end{itemize}

\subsection{\textbf{Models of $\mathbf{H}_{\mathbb{H}}^m$.}}
As for real and complex hyperbolic space, there are many models for quaternionic hyperbolic space.  We identify three models here.  

\begin{enumerate}[\quad(M1)]
\item The \textbf{negative $\mathbb{H}$-line model} is the set $\mathbf{L}_{\mathbb{H}}^m:=\left\{[\mathbf{v}]\in \mathbb{P}(V)\ | \ h(\mathbf{v},\mathbf{v})<0 \right\}$.

\item The \textbf{ball model} is the set $\mathbf{B}_{\mathbb{H}}^m:=\left\{\mathbf{v}=(v_1,v_2,\ldots,v_{m},1)\in \mathbb{H}^{m,1}\ \bigg| \  \sum_{i=1}^m \overline{v_i}v_i<1 \right\}.$

\item The \textbf{homogenous space model}  is the quotient 
$(\mathbf{Sp}(m).\mathbf{Sp}(1))\backslash \mathbf{Sp}(m,1)$.
\end{enumerate}

\qquad
It is not hard to see that these spaces are all diffeomorphic.
The projectivisation map sending $\mathbf{v}$ to $[\mathbf{v}]$ is a diffeomorphism between $\mathbf{B}_{\mathbb{H}}^m$ and $\mathbf{L}_{\mathbb{H}}^m$.
The group $\mathbf{Sp}(m,1)$ acts transitively on negative lines in $\mathbb{H}^{m,1}$ \cite[\S20]{Mostow}
and $K:=\mathbf{Sp}(m).\mathbf{Sp}(1) =\begin{pmatrix} \mathbf{Sp}(m) & 0 \\ 0 & \mathbf{Sp}(1)\end{pmatrix}\subset \mathbf{Sp}(m,1)$ is the stabilizer of the negative line $[\mathbf{0},1]$.  
The continuous bjiection from $(\mathbf{Sp}(m).\mathbf{Sp}(1))\backslash \mathbf{Sp}(m,1)$ to  $\mathbf{L}_{\mathbb{H}}^m$ is in fact a homeomorphism, and hence a diffeomorphism \cite[II.4.3]{H}.

\subsection{\textbf{Riemannian and Quaternion-Hermitian Metrics on $\mathbf{H}_{\mathbb{H}}^m$}}
There are two Riemannian metrics of interest on $\mathbf{H}_{\mathbb{H}}^m$: the \textbf{hyperbolic metric}, which has sectional curvature bounded between $-1$ and $-\frac{1}{4}$, and the \textbf{Killing  metric}, which is the induced metric from the Lie group $\mathbf{Sp}(m,1)$.  
Both of these Riemannian metrics yield $\mathbf{Sp}(m,1)$-invariant distance formulas, and it is known that such metrics on $\mathbf{H}_{\mathbb{H}}^m$ are unique up to scaling \cite[\S20]{Mostow}.
In this subsection we describe these two metrics and explicitly compute the scaling factor between them.

\subsubsection{The Hyperbolic Metric}
For a negative vector $\mathbf{v}\in \mathbb{H}^{m,1}$, the tangent space $T_{[\mathbf{v}]}(\mathbf{L}_{\mathbb{H}}^m)$ can be identified with the vector subspace 
$[\mathbf{v}]_{\mathbf{v}}^\perp:=\{\mathbf{w}\in T_{\mathbf{v}}(\mathbb{H}^{m,1})\cong\mathbb{H}^{m,1}\ | \ h(\mathbf{w},\mathbf{v})=0\}$
by scaling by the \textbf{norm} $||\mathbf{v}||:=\sqrt{-h(\mathbf{v},\mathbf{v})}$ of $\mathbf{v}$, 
i.e. there is a natural map
$$\rho_{\mathbf{v}}:[\mathbf{v}]_{\mathbf{v}}^\perp \to T_{[\mathbf{v}]}(\mathbf{L}_{\mathbb{H}}^m),$$
$$\mathbf{w}\mapsto\frac{\mathbf{w}}{||\mathbf{v}||}.$$

Via these identifications, the Hermitian form $h$ naturally defines a quaternion-Hermitian metric on $\mathbf{L}_{\mathbb{H}}^m$ as we now show.
Let $L\in \mathbf{L}_{\mathbb{H}}^m$ and let $W_1, W_2\in T_{L}(\mathbf{L}_{\mathbb{H}}^m)$.  By the above remarks, there exists a negative vector $\mathbf{v}$ such that $L=[\mathbf{v}]$ and tangent vectors $\mathbf{w}_1, \mathbf{w}_2\in[\mathbf{v}]_{\mathbf{v}}^\perp$ such that
$\rho_{\mathbf{v}}(\mathbf{w}_1)=W_1$ and $\rho_{\mathbf{v}}(\mathbf{w}_2)=W_2$.  
Then the quaternion-Hermitian metric at $L$ is given by
\begin{align}
\notag \langle W_1,W_2\rangle_L&:=\langle \mathbf{w}_1,\mathbf{w}_2\rangle_{\mathbf{v}}\\
\notag&:= 4h\left(\frac{\mathbf{w}_1}{\sqrt{-h(\mathbf{v},\mathbf{v})}},\frac{\mathbf{w}_2}{\sqrt{-h(\mathbf{v},\mathbf{v})}}\right)\\
\label{eqn:hypmetric}&=\frac{-4h(\mathbf{w}_1,\mathbf{w}_2)}{h(\mathbf{v},\mathbf{v})}.
\end{align}
The real part of this form defines a Riemannian metric $g$ on $\mathbf{L}_{\mathbb{H}}^m$ which we call the \textbf{hyperbolic metric}, i.e., 

\begin{align}\label{eqn:hypmetric2}
g_L(W_1,W_2):=\mathrm{Re}(\langle W_1,W_2\rangle_L).
\end{align}

\qquad More generally, arbitrary vectors $\mathbf{w}_1, \mathbf{w}_2\in T_{\mathbf{v}}(\mathbb{H}^{m,1}) \cong \mathbb{H}^{m,1}$ represent tangent vectors in $T_{[\mathbf{v}]}(\mathbf{L}_{\mathbb{H}}^m)$ by projecting onto $[\mathbf{v}]_{\mathbf{v}}^\perp$ and we define 
\begin{align}
\notag  \langle \mathbf{w}_1,\mathbf{w}_2\rangle_{\mathbf{v}}&:= \frac{-4h(\mathbf{w}_1-\mathrm{proj}_{\mathbf{v}}\mathbf{w}_1,\mathbf{w}_2-\mathrm{proj}_{\mathbf{v}}\mathbf{w}_2)}{h(\mathbf{v},\mathbf{v})} \\ 
\label{metricrep}&=-4\left(\frac{h(\mathbf{v},\mathbf{v})h(\mathbf{w}_1,\mathbf{w}_2)-h(\mathbf{w}_1,\mathbf{v})h(\mathbf{v},\mathbf{w}_2)}{h(\mathbf{v},\mathbf{v})^2}\right)
\end{align}

where $\mathrm{proj}_{\mathbf{v}}\mathbf{w}$ is the projection of $\mathbf{w}$ onto $[\mathbf{v}]$ given by

$$\mathrm{proj}_{\mathbf{v}}\mathbf{w}=\mathbf{v}\frac{h(\mathbf{v},\mathbf{w})}{h(\mathbf{v},\mathbf{v})}.$$

This is similar to the complex hyperbolic case (see \cite[\S2]{Epstein}) but note that we have normalized the metric to fix the sectional curvatures between $-1$ and $-\frac{1}{4}$.

\qquad This metric gives the \textbf{hyperbolic distance} between two negative vectors  $\mathbf{v}_1,\mathbf{v}_2\in \mathbb{H}^{m,1}$ via the following formula 

\begin{align}\label{eqn:hypdist}
\mathrm{dist}(\mathbf{v}_1,\mathbf{v}_2)=2\cosh^{-1}\left( \sqrt{\frac{h(\mathbf{v}_1,\mathbf{v}_2)h(\mathbf{v}_2,\mathbf{v}_1)}{h(\mathbf{v}_1,\mathbf{v}_1)h(\mathbf{v}_2,\mathbf{v}_2)}}\right).
\end{align}

This induces well defined compatible distance functions on both $\mathbf{B}_{\mathbb{H}}^m$ and $\mathbf{L}_{\mathbb{H}}^m$.
Furthermore, it is clear that this distance formula is $\mathbf{Sp}(m,1)$-invariant.
Again, as in the complex hyperbolic case (see \cite[3.1.7]{Goldman}), this normalization will fix the sectional curvatures between $-1$ and $-\frac{1}{4}$.  

\subsubsection{The Killing Metric}
 The tangent space $T_{[\mathbf{0},1]}(\mathbf{L}_{\mathbb{H}}^m)$ is naturally identified with the ``horizontal subspace'' $V_0:=\{(\mathbf{w},0)\}\subset T_{(\mathbf{0},1)}(\mathbb{H}^{m,1})\cong \mathbb{H}^{m,1}$ which in turn is naturally identified with the $\mathbb{H}$-subspace $\mathfrak{p}$ coming from the Cartan decomposition of $\mathfrak{sp}(m,1)$ via the map
$$T:V_0\to \mathfrak{p}\qquad \mathbf{w}\mapsto \begin{bmatrix} 0_{n\times n}& \mathbf{w}\\ {}^T\overline{\mathbf{w}} & 0\end{bmatrix}.$$
%
%
The adjoint representations of $\mathbf{Sp}(m,1)$ and $\mathfrak{sp}(m,1)$ are defined by
\begin{align*}
&\mathrm{Ad}:\mathbf{Sp}(m,1)\to \mathrm{Aut}_\R(\mathfrak{sp}(m,1))\subset \mathbf{GL}_{2m^2-5m+3}(\R),
&&\mathrm{Ad}(x)(Y):= xYx^{-1},\\
&\mathrm{ad}:\mathfrak{sp}(m,1)\to \mathrm{End}_\R(\mathfrak{sp}(m,1))\subset \mathbf{Mat}_{2m^2-5m+3}(\R), &&
\mathrm{ad}(X)(Y):= XY-YX,
\end{align*}
where the matrix multiplication comes from the matrix multiplication of $\mathrm{Mat}_{m+1}(\mathbb{H})$.  
The \textbf{Killing form} on $\mathfrak{p}$ is a bilinear $\R$-form defined via
$$\kappa_0(X,Y):=\mathrm{Tr}((\mathrm{ad}X)(\mathrm{ad}Y)).$$
The \textbf{Killing metric} $\kappa$ on $\mathbf{H}_{\mathbb{H}}^m$ is the Riemannian metric obtained by using the action of $\mathbf{Sp}(m,1)$ to identify each tangent space with the inner product space $(V_0, \kappa_0)$.
It is an immediate consequence that the corresponding distance formula is $\mathbf{Sp}(m,1)$-invariant.

\subsubsection{The Scaling Factor}
While the hyperbolic metric is geometric in nature and the Killing metric is Lie theoretic in nature, they both yield
$\mathbf{Sp}(m,1)$-invariant metrics on $\mathbf{H}_{\mathbb{H}}^m$, and hence there is some constant $c\in\R_{>0}$ such that $cg=\kappa$.

\qquad Let $h_0$ denote the restriction of $h$ to $V_0$ and let $g_0$ denote the bilinear form on $V_0$ obtained by evaluating the hyperbolic metric $g$ at  $[\mathbf{0},1]$.

\begin{lem}\label{lemma:gscale}
Let $\mathbf{w}\in \mathbb{H}^n$ be the vector with 1 in the first entry and zeros elsewhere.  Then
$g_0(\mathbf{w},\mathbf{w})=4$.
\end{lem}

\begin{proof}
A direct computation using \eqref{eqn:hypmetric} gives $$g_0(\mathbf{w},\mathbf{w})=\mathrm{Re}\left(\frac{-4h_0(\mathbf{w},\mathbf{w})}{-1}\right)=\mathrm{Re}(4(1+0+0+\cdots +0))=4.$$
\end{proof}

The following two lemmas about the Lie algebra $\mathfrak{sp}(m,1)$ are easily verified with a little Lie algebra bookkeeping.

\begin{lem}
The following is an $\R$-vector space basis of $\mathfrak{sp}(m,1)$.
\begin{itemize}
\item $X_\ell(\alpha)=\begin{bmatrix} 0_{n\times n}& \mathbf{w}\\ {}^T\overline{\mathbf{w}} & 0\end{bmatrix}$ where $\mathbf{w}$ is zero everywhere except $\alpha\in\{1,i,j,k\}$ in its $\ell^{th}$ entry.
There are $4m$ of these.
\item $Y_{\ell_1\ell_2}(\alpha)$, $\ell_1<\ell_2<m+1$ is zero everywhere except $\alpha\in \{1,i,j,k\}$ in the $\ell_1\ell_2^{th}$ entry and $-\overline{\alpha}$ in the $\ell_2\ell_1^{th}$ entry.
There are $4\left(\frac{(m-1)m}{2}\right)=2m^2-2m$ of these.
\item $H_\ell(\alpha)$ is zero everywhere except $\alpha\in \{i,j,k\}$ in the $\ell^{th}$ diagonal entry. 
There are $3m+3$  of these.
\end{itemize}
Furthermore $\{X_\ell(\alpha)\}$ is a basis for $\mathfrak{p}$ and $\{Y_{\ell_1\ell_2}(\alpha), H_\ell(\alpha)\}$ is a basis for $\mathfrak{k}$.
\end{lem}
%
%
%

\begin{lem}${}$
\begin{enumerate}
\item $[X_{\ell_1}(\alpha_1),X_{\ell_2}(\alpha_2)]=Y_{\ell_1\ell_2}(\alpha_1\overline{\alpha_2})$ for $\ell_1<\ell_2$,
\item $[X_{\ell_1}(\alpha_1),Y_{\ell_1\ell_2}(\alpha_2)]=X_{\ell_2}(\overline{\alpha_2}\alpha_1)$ for $\ell_1<\ell_2$,
\item $[X_{\ell_1}(\alpha_1),H_{\ell_2}(\alpha_2)]=0$,
\item $[X_{\ell_1}(\alpha_1),Y_{\ell_2\ell_3}(\alpha_2)]=0$ where $\ell_1\ne\ell_2<\ell_3$.
\end{enumerate}
\end{lem}

\begin{cor}\label{cor:kscale}
$\kappa_0(X_{1}(1),X_{1}(1))=8(m-1)$.
\end{cor}


\begin{proof}
We use the above lemmas to compute $\mathrm{ad}(X_1(1))$.  It is not hard to see that $\mathrm{ad}(X_1(1))$ is going to be symmetric with all zeros except for $4(m-1)$ 1's from $X_\ell(\alpha)$ for $\ell>1$ and $4(m-1)$ 1's from $Y_{\ell_1\ell}(\alpha)$ for $\ell>1$.
Hence $(\mathrm{ad}(X_1(1)))^2$ is a diagonal matrix with $8(m-1)$ 1's along the diagonal and all the rest 0's.
\end{proof}

Corollary \ref{cor:kscale} together with Lemma \ref{lemma:gscale} and earlier remarks prove the following proposition.

\begin{prop}\label{samemetrics}
Let $g$ and $\kappa$ denote the hyperbolic metric and Killing metric on $\mathbf{H}_{\mathbb{H}}^m$ respectively.   Then
$2(m-1)g=\kappa$.
\end{prop}


%
%
%
%
%
%
%
%

\subsection{\textbf{Isometries of $\mathbf{H}_{\mathbb{H}}^m$.}} Given equations \eqref{eqn:hypmetric} and \eqref{eqn:hypdist}, it is clear that $\mathbf{Sp}(m,1)$ acts isometrically on $\mathbf{H}_{\mathbb{H}}^m$ and it in fact the full isometry group is given by
\begin{align}
\mathrm{Isom}(\mathbf{H}_{\mathbb{H}}^m)\cong \mathbf{PSp}(m,1):=\mathbf{Sp}(m,1)/\{\pm I\}.
\end{align}

Note that all isometries in $\mathbf{PSp}(m,1)$ are orientation preserving.

\subsection{\textbf{Totally Geodesic Subspaces of $\mathbf{H}_{\mathbb{H}}^m$.}} 
Proposition \ref{samemetrics} says that totally geodesic subspaces of $\mathbf{H}_{\mathbb{H}}^m$ coming from the hyperbolic metric and the Killing metric coincide.  
We may now use Lie theory to classify the totally geodesic subspaces of $\mathbf{H}_{\mathbb{H}}^m$.  
The totally geodesic subspaces are naturally in bijective correspondence with Lie triple systems of $\mathfrak{p}$ \cite[IV. Thm 7.2]{H}, i.e. $\R$-vector subspaces $\mathfrak{m}\subset \mathfrak{p}$ such that $[[\mathfrak{m},\mathfrak{m}], \mathfrak{m}]\subset \mathfrak{m}]$.
For $X,Y, Z\in \mathfrak{p}$, let

$$X = \begin{bmatrix}0_{n\times n}& \mathbf{v}\\ {}^T\overline{\mathbf{v}} & 0\end{bmatrix}
\qquad
Y = \begin{bmatrix} 0_{n\times n}& \mathbf{w}\\ {}^T\overline{\mathbf{w}} & 0\end{bmatrix}
\qquad
Z= \begin{bmatrix} 0_{n\times n}& \mathbf{u}\\ {}^T\overline{\mathbf{u}} & 0\end{bmatrix}
$$

where $\mathbf{v}, \mathbf{w}, \mathbf{u}\in \mathbb{H}^m$.
A quick computation gives

\begin{align*}
[[X,Y],Z]&=\begin{bmatrix} 0_{n\times n}& \boldsymbol\alpha\\ {}^T\overline{\mathbf{\boldsymbol\alpha}} & 0\end{bmatrix}
\end{align*}

where 

\begin{align}\label{lietrip}
\boldsymbol\alpha &= \mathbf{v}({}^T\overline{\mathbf{w}} \mathbf{u}) - \mathbf{w}({}^T\overline{\mathbf{v}}\mathbf{u})-\mathbf{u}({}^T\overline{\mathbf{v}}\mathbf{w}-{}^T\overline{\mathbf{w}}\mathbf{v}) \notag\\
&= \mathbf{v} h_0(\mathbf{w}, \mathbf{u}) - \mathbf{w} h_0(\mathbf{v},\mathbf{u})-\mathbf{u}(h_0( \mathbf{v},\mathbf{w})-h_0(\mathbf{w},\mathbf{v})).
\end{align}
%
From this computation, we see there are three classes of Lie triple subspaces $W_0\subset 
V_0\cong \mathfrak{p}$:

\phantomsection\label{totallydef}
\begin{enumerate}[\quad (T1)]
\item \textbf{Totally Real}.  
$h_0( W_0, W_0)\subset \R$.   
It follows that $W_0$ and $W_0\delta$ are $g_0$-orthogonal for any pure quaternion $\delta\in \mathbb{H}$.
\item \textbf{Totally Complex}.  
$W_0$ is not totally real but there exists a pure quaternion $\delta\in \mathbb{H}$ such that  $h_0(W_0, W_0)\subset \R(\delta)$.  
It follows that this condition is equivalent to the criterion that $W_0\delta=W_0$ and $W_0\mu$ and $W_0$ are $g_0$-orthogonal where $\mu$ is a pure quaternion complimentary to $\delta$ 
(i.e., $\mathbb{H}$ is generated as an $\R$-algebra by $\delta$ and $\mu$, where $\delta\mu=-\mu\delta$). Identifying the quadratic extension $\R(\delta)$ with $\C$, then $W_0\alpha=W_0$ for any $\alpha\in \mathbb{C}$.
\item \textbf{Totally Quaternionic}.  $W_0$ is neither totally real nor totally complex.  
In this case, it follows that for any pure quaternion $\delta$ and a compliment $\mu$, $W_0\delta=W_0=W_0\mu$.  
In particular, $W_0\alpha=W_0$ for any $\alpha\in \mathbb{H}$.
\end{enumerate}

Conversely, it is not hard to see from \eqref{lietrip} that any totally real, totally complex, or totally quaternionic subspace $W_0\subset \mathbb{H}^n$ determines a Lie triple system, and hence is the tangent space to a totally geodesic submanifold of quaternionic hyperbolic space going through $[\mathbf{0},1]$.

\begin{Def}
An $\R$-subspace $W\subset V$ \textbf{totally real} (resp. \textbf{totally complex}, \textbf{totally quaternionic}) if it is in the $\mathbf{Sp}(m,1)$-orbit of a subspace containing $[\mathbf{0},1]$ whose ``horizontal component'' is totally real (resp. totally complex, totally quaternionic).
\end{Def}

%
%

We now have the following classification.

\begin{thm}[Classification of Totally Geodesic Subspaces of $\mathbf{H}_{\mathbb{H}}^m$]\label{totgeoclassification}
There is a bijection between totally geodesic subspaces $N$ of $\mathbf{L}_{\mathbb{H}}^m$ and $\R$-subspaces $W\subset V$ containing a negative vector.
Furthermore, the induced Riemannian metric on $N$ is determined by $W$ as follows:\\
\begin{tabular}{llcl}
\quad\quad(1)&$N$ is real hyperbolic  &$\Leftrightarrow$ &$W$ is totally real.\\
\quad\quad(2)&$N$ is complex hyperbolic  &$\Leftrightarrow$ &$W$ is totally complex.\\
\quad\quad(3)&$N$ is quaternionic hyperbolic  &$\Leftrightarrow$ &$W$ is totally quaternionic.
\end{tabular}
\end{thm}


\begin{proof}
If $W\subset V$ is an $\R$-subspace containing a negative vector, then $N:=[W]\cap \mathbf{L}_{\mathbb{H}}^m$ is totally geodesic submanifold of $\mathbf{L}_{\mathbb{H}}^m$ and all totally geodesic subspaces arise in this way.
When $W$ is totally real (resp. totally complex, totally quaternionic), the quaternion-Hermitian metric on  $\mathbf{L}_{\mathbb{H}}^m$ given in \eqref{eqn:hypmetric} induces a Riemannian (resp. Hermitian, quaternion-Hermitian) metric on $N$ and the restriction of the hyperbolic metric \eqref{eqn:hypmetric2} to $N$ then
assumes familiar form of the hyperbolic metric on the negative line model of real (resp. complex, quaternionic) hyperbolic space.  
Conversely, when the induced Riemannian metric on $N$ is quaternionic (resp. complex, real) hyperbolic, the tangent bundle comes equipped with almost quaternionic (resp. almost complex, no complex) structure, which implies $W$ is totally quaternionic (resp. totally complex, totally real).
\end{proof}


\section{Background on Hermitian Forms over Division Algebras}

\qquad To construct lattices in $\mathbf{Sp}(m,1)$, and more generally in Lie groups of type $C_n$, we use Hermitian forms over division algebras over number fields.  In this section we introduce the algebra of these objects which we use in later sections.
We introduce the following notation:
\begin{itemize}
\item $k$ is a number field with places $V_k$ and infinite places $V_k^\infty$,
\item $\mathcal{O}_k$ is its ring of integers,
\item $D=\left(\frac{a,b}{k}\right)$ is a quaternion algebra with center $k$ and standard involution denoted by $\alpha\mapsto \overline{\alpha}$,
\item $\varphi_D=\langle 1,-a,-b,ab\rangle$ is the norm form of $D$,
\item $V$ is an $n$-dimensional vector space over $D$,
\item $h:V\times V\to D$ is a Hermitian form relative to the standard involution of $D$,
\item for each $v\in V_k$, $(V_v,h_v)$ is the extension of $h$ to $V_v:=V\otimes_k k_v$.
\end{itemize}

Recall that since $h$ is Hermitian it satisfies $\overline{h(x,y)}=h(x,y)$.  
For all $x\in V$, $q_h(x):=h(x,x)=\overline{h(x,x)}$ and hence $q_h(x)\in k$.  
Viewing $V$ as a $4n$-dimensional vector space over $k$, $(V,q_h)$ is a quadratic space and $q_h$ is called the \textbf{trace form} of $h$.  
Given any orthogonal basis, $h$ can be represented by a diagonal matrix $\mathrm{diag}( a_1,\ldots, a_n)$ with entries in $k$.
Let $q$ denote the $n$-dimensional quadratic form $q=\langle a_1,\ldots, a_n\rangle$ over $k$, which we  call the \textbf{restriction form} of $h$ to $k$.
A quick computation shows  $q_h= \varphi_D\otimes q$.  %
We will frequently make use of the fact that two Hermitian forms $h_1$ and $h_2$ over $D$ are isometric as Hermitian forms over $D$ if and only if $q_{h_1}$ and $q_{h_2}$ are isometric as quadratic forms over $k$ \cite[Theorem 10.1.7]{Sch}.

\qquad 
For a nonarchimedean place $v\in V_k$, we let $(a,b)$ denote the Hilbert symbol of $a,b\in k_v$.
A quadratic form over a nonarchimedean field is completely determined by three invariants: its dimension ($\dim$), determinant, ($\det$), and Hasse--Minkowski invariant ($c$) \cite[Theorem 63:23]{OM}.  

\begin{lem}\label{localforms}
Let $v\in V_k$ be a nonarchimedean place and let $D=\left(\frac{a,b}{k_v}\right)$ be a quaternion division algebra over $k_v$.  
If $h_v$ is an $m$-dimensional Hermitian form over $D$ and $q_{h_v}$ is its trace form,
then
\begin{align*}
\dim q_{h_v} &= 4\dim h_v,\\
\det q_{h_v} &= 1,\\
c(q_{h_v}) &= (a,b)^m(-1,-1)^m.
\end{align*}
In particular,  $c(q_{h_v})=1$ when $m$ is even.
\end{lem}

\begin{proof} We directly compute each.
\begin{align*}
\dim (q_{h_v})&=\dim(\varphi_D\otimes q)\\
&=\dim(\langle 1, -a, -b, ab\rangle \otimes q)\\
&=\dim\left(\langle 1\rangle q\oplus \langle -a\rangle q\oplus \langle -b\rangle q\oplus \langle ab\rangle q\right)\\
&= 4\dim h_v\\
& \\\det (q_{h_v})&=\det(\varphi_D\otimes q)\\
&=\det(\langle 1, -a, -b, ab\rangle \otimes q)\\
&=\det\left(\langle 1\rangle q\oplus \langle -a\rangle q\oplus \langle -b\rangle q\oplus \langle ab\rangle q\right)\\
&=(\det q)((-a)^m\det q)((-b)^m\det q)((ab)^m\det q)\\
&= 1\\
& \\
\pagebreak
c(q_{h_v})&=c(\varphi_D\otimes q)\\
&=c(\langle 1, -a, -b, ab\rangle \otimes q)\\
&=c\left(\langle 1\rangle q\oplus \langle -a\rangle q\oplus \langle -b\rangle q\oplus \langle ab\rangle q\right)\\
&=c(q)\ c( \langle -a\rangle q\oplus \langle -b\rangle q\oplus \langle ab\rangle q)\ (d, d)\\
&=c(q)\ c( \langle -a\rangle q)\ c(\langle -b\rangle q\oplus \langle ab\rangle q)\ ((-a)^md, (-a)^m)\ (d, d)\\
&=c(q)\ c( \langle -a\rangle q)\ c(\langle -b\rangle q) \ c(\langle ab\rangle q)\ ((-b)^md, (ab)^md)\ ((-a)^md, (-a)^m)\ (d, d)\\
&=c(q)\ c( \langle -a\rangle q)\ c(\langle -b\rangle q) \ c(\langle ab\rangle q)\ (d, d) (d, (-a)^m) ((-b)^m,a^m) \ ((-a)^md, (-a)^m)\ (d, d)\\
&=c(q)\ c( \langle -a\rangle q)\ c(\langle -b\rangle q) \ c(\langle ab\rangle q)\  ((-b)^m,a^m) \ ((-a)^m, (-a)^m)\\&=c(q)\ c( \langle -a\rangle q)\ c(\langle -b\rangle q) \ c(\langle ab\rangle q)\  (a^m, b^m) , (a^m, (-1)^m) \ ((-1)^m, (-1)^m) (a^m, a^m)\\
&=c(q)\ c( \langle -a\rangle q)\ c(\langle -b\rangle q) \ c(\langle ab\rangle q)\  (a, b)^{m^2}(-1, -1)^{m^2}\\
&=(c(q))^4 (-a,e) (-b,e) (ab,e)\  (a, b)^{m^2}(-1, -1)^{m^2}\\
&=(a, b)^m(-1, -1)^m
\end{align*}
\end{proof}

\qquad
 Lemma \ref{localforms} allows us to give a proof of  the following known result (see \cite[10.1.8 (ii)]{Sch}).

\begin{prop}\label{uniquelocalHermitian}
Over a nonarchimidean local field, there exists a unique Hermitian form of dimension $m$ over the unique quaternion division algebra.
\end{prop}

\begin{proof}
The local invariants $\dim$, $\det$, and $c$ of a trace form $q_{h_v}$ are totally determined by $m$ and $D$.
Therefore the trace forms of two $m$-dimensional Hermitian forms over $D$ are isometric \cite[Theorem 63:23]{OM}, and hence the Hermitian forms themselves are isometric \cite[Theorem 10.1.7]{Sch}.
\end{proof}


\section{Constructions of Irreducible Lattices}\label{section:construction}

\qquad 
There are three flavors of noncompact building block Lie groups of type $C_n$ corresponding to the three types \cite{TitsClassification} of $\R$-simple, $\R$-isotropic, algebraic $\R$-groups of type $C_n$:
\begin{enumerate}
\item $\mathbf{Sp}_{2n}(\R)$ - The split form over $\R$.  It comes from a $2n$-dimensional symplectic form over $\R$.  This group has $\R$-rank $n$.
\item  $\mathbf{Sp}(p,n-p)$, $0< p< n$ - The non-split forms over $\R$.  They come from $n$-dimensional Hermitian form over $\mathbb{H}$.  This group has $\R$-rank $\min\{p,n-p\}$.
\item  $\mathbf{Sp}_{2n}(\C)$ - The split form over $\C$.  It comes from a $2n$-dimensional symplectic form over $\C$.  This group has $\R$-rank $2n$.
\end{enumerate}
A \textbf{group of type $C_n$ (with no compact factors)} is isogenous to a product  $$G=(\mathbf{Sp}_{2n}(\R))^q \times \prod_{i=1}^r \mathbf{Sp}(p_i,n-p_i) \times (\mathbf{Sp}_{2n}(\C))^s$$ where for each $i$, $0<p_i<n$, and $p,q,r$ are nonnegative integers.
In this section, our goal is to give constructions of the irreducible lattices in groups of type $C_n$, $n\ge 3$.
Recall that a lattice in such a product is \textbf{irreducible} if it is not commensurable to a product of lattices coming from factors \cite{Raghunathan}.
All such irreducible lattices in $G$ are arithmetic \cite{Mar} \cite{GS}.
As such we may use the classification of algebraic groups of type $C_n$ over number fields \cite{TitsClassification} to construct all irreducible lattices.



\qquad There are two types of lattices, depending on whether the initial $k$-group is $k$-split.  
While we give both constructions, the nonsplit-type lattices will be of more interest to us since they are the lattices that give the fundamental groups of finite volume quaternionic hyperbolic orbifolds.

\begin{con}[Split-type Irreducible Lattices]\label{arithmeticsplitnlattice}${}$
\begin{enumerate}
\item Let $\mathbf{G}=\mathbf{Sp}_{2n}(k)$ and $G=(\mathbf{Sp}_{2n}(\R))^q  \times (\mathbf{Sp}_{2n}(\C))^s$ where
	\begin{itemize}
		\item $q$ is the number of real places of $k$ and
		\item $s$ is the number of complex places of $k$.
	\end{itemize} 
	Observe $G$ is a semisimple Lie group with no compact factors and is simple when $q=1$ and $s=0$.
\item Via the diagonal embedding, $\mathbf{Sp}_{2n}(\mathcal{O}_k)$ sits inside $G$ as an arithmetic lattice.
\item Let $\Gamma\subset G$ be a subgroup commensurable up to $G$-automorphism with $\mathbf{Sp}_{2n}(\mathcal{O}_k)$.  Then  $\Gamma$ is an \textbf{irreducible lattice of split-type $C_n$}.
\item Let  $K\subset G$ be a maximal compact subgroup and let $M_{\Gamma}:=\Gamma \backslash G /K$.  Then
\begin{enumerate}
\item $M_\Gamma$ is a \textbf{locally symmetric orbifold of split-type $C_n$.}   
\phantomsection
\label{fieldalgebradef2}
\item  $k(M_\Gamma):=k$ is the \textbf{field of definition} of $M_{\Gamma}$.  
\end{enumerate}
\end{enumerate}
\end{con}

\begin{con}[Nonsplit-type Irreducible Lattices]\label{arithmeticcnlattice}${}$
\begin{enumerate}
 \item Let $\mathbf{G}:=\mathbf{SU}(V,h)$ be the absolutely almost simple $k$-group defined by $(V,h)$ and let $SU(h):=\mathbf{G}(k)$.
\item Let $(m_+^{(v)}, m_-^{(v)})$ denote the signature of $(V_v,h_{v})$ for each real $v\in V_k^\infty$ where $D$ ramifies.
 \item Let $\mathbf{G}_{v}$ denote the algebraic $k_{v}$-group $\mathbf{SU}(V_{v},h_{v})$ for each  $v\in V_k^\infty$.
\begin{itemize}
\item If $v$ is real and $D_v$ splits, then $\mathbf{G}_{v}(\R)\cong \mathbf{Sp}_{2n}(\R)$.  
\item If $v$ is real and $D_v$ ramifies, then $\mathbf{G}_{v}(\R)\cong \mathbf{Sp}(m_+^{(v)}, m_-^{(v)})$.  
\item If $v$ is complex, then $\mathbf{G}_{v}(\R)\cong \mathbf{Sp}_{2n}(\C)$. 
\end{itemize}
\item Let $\mathbf{G}':=R_{k/\Q}\mathbf{G}$  be the semisimple $\Q$-group formed by restriction of scalars.   Then $\mathbf{G}'(\R)=\prod \mathbf{G}_{v}(\R)$ is a semisimple Lie group which has compact factors at precisely the real places where $h$ is anisotropic.  
Furthermore, there is an isomorphism $\mathbf{G}(k)\cong \mathbf{G}'(\Q)$, and hence there is a natural diagonal embedding $SU(h)\to  \mathbf{G}'(\R)$.
\item Let $G$ be the projection of $\mathbf{G}'(\R)$ onto its noncompact factors and denote the projection map by $\pi:\mathbf{G}'(\R)\to G$.   Observe that $G$ is a semisimple Lie group with no compact factors and is simple when $q+r=1$ and $s=0$.
\item Fix an order $\mathcal{O}_D$ in $D$ and an $\mathcal{O}_D$-lattice $L\subset V$,  and let $G_L=\{T\in \mathbf{G}(k) \ | \ T(L)\subset L\}$.
Then $G_L$ sits as a discrete arithmetic subgroup of $\mathbf{G}'(\R)$.
\item Let $\Gamma\subset G$ be a subgroup commensurable up to $G$-automorphism with $\pi(G_L)$.  Then $\Gamma$ is an \textbf{irreducible lattice of nonsplit-type $C_n$}.
\item Letting
	\begin{itemize}
		\item $q$ be the number of real places of $k$ where $D$ splits,
		\item $r$ be the number of real places of $k$ where $D$ ramifies and $h$ is isotropic,  
		\item $s$ be the number of complex places of $k$, and
		\item $p_i=m_+^{(v_i)}$ where $\{v_1, \ldots, v_r\}$ is the set of real places where $D$ ramifies and $h$ is isotropic,
	\end{itemize} 
we have the following diagram illustrating our construction of irreducible arithmetic lattices in $G$:
\begin{center}
\begin{tikzpicture}[scale=1.5]
\node (A) at (-1,1) {$SU(h)$};
\node (B) at (4.5,1) {$\left(
\displaystyle\prod\limits_{\substack{v \ \mathrm{real} \\ D\ \mathrm{splits}}} \mathbf{Sp}_{2n}(\R) \quad \times 
\displaystyle\prod\limits_{\substack{v \ \mathrm{real} \\ D\ \mathrm{ramifies}}} \mathbf{Sp}(m_+^{(v)}, m_-^{(v)}) \quad \times 
\displaystyle\prod\limits_{v\ \mathrm{complex}}\mathbf{Sp}_{2n}(\C)\right)$};
\node (C) at (-2.5,1) {$G_L$};
\node (D) at (4.5,-1.5) {$G= (\mathbf{Sp}_{2n}(\R))^q  \times 
\prod_{i=1}^r \mathbf{Sp}(p_i,n-p_i)  \times 
(\mathbf{Sp}_{2n}(\C))^s$};
\node (E) at (-2.5,-1.5) {$\Gamma$};
\path[right hook->,font=\scriptsize,>=angle 90]
(A) edge node[above]{diagonal} (B)
(C) edge node[above]{} (A)
(E) edge node[below]{with $\pi(G_L)$}  (D)
(E) edge node[above]{Commensurable (up to $G$-automorphism)}  (D);
\path[->,font=\scriptsize,>=angle 90]
(C) edge node[below]{lattice} (D)
(B) edge node[right]{$\pi$} (D);
\end{tikzpicture}
\end{center}	
\item Let  $K\subset G$ its maximal compact subgroup and let $M_{\Gamma}:=\Gamma \backslash G /K$.  Then
\begin{enumerate}
	\item $M_\Gamma$ is a \textbf{locally symmetric orbifold of nonsplit-type $C_n$},
	\phantomsection
	\label{fieldalgebradef}
	\item $k(M_\Gamma):=k$ is the \textbf{field of definition} of $M_{\Gamma}$, and
	\item  $D(M_{\Gamma}):=D$ is the \textbf{algebra of definition} of $M_{\Gamma}$.  
\end{enumerate} 
\end{enumerate}
\end{con}
A choice of another order $\mathcal{O}_D'$ in $D$, another $\mathcal{O}_D'$-lattice $L'\subset V$, and another group $\Gamma'$ commensurable up to $G$-automorphism with $\pi(G_{L'})$ will produce a space $M_{\Gamma'}$ which is commensurable with $M_{\Gamma}$. 
Hence choosing $h$ determines a commensurability class of orbifolds which we denote by $M_h$.

\begin{rem}\label{rem:minfielddef}
By \cite[Lemma 2.6]{PrasadRapinchukWC}, when $M_\Gamma$ is simple, $k(M_\Gamma)$ coincides with the minimal field of definition of $\Gamma$ in the sense of Vinberg \cite{Vinberg}.
\end{rem}

 \begin{rem}\label{rem:galoisaction}
The above two constructions show that one obtains a commensurability class by choosing $k$ and $\mathbf{G}$.
If $k'$ and $\mathbf{G}'$ also give the same commensurability class, by \cite[Proposition 2.5]{PrasadRapinchukWC}, there is a field isomorphism $\tau:k'\to k$ and a $k$-rational isomorphism 
$\phi:\mathbf{G}\to \mathbf{G}'\times_{\mathrm{Spec}\, k'}\mathrm{Spec}\, k$ where $\mathbf{G}'\times_{\mathrm{Spec}\, k'}\mathrm{Spec}\, k$  is the associated base change from $k'$ to $k$.
\begin{center}
\begin{tikzpicture}[scale=1.5]
\node (A) at (-1,1) {$\mathbf{G}'\times_{\mathrm{Spec}\, k'}\mathrm{Spec}\, k$};
\node (B) at (1.5,1) {$\mathbf{G}'$};
\node (C) at (-1,0) {$\mathrm{Spec}\, k$};
\node (F) at (-3.5,0) {$\mathrm{Spec}\, k$};
\node (D) at (1.5,0) {$\mathrm{Spec}\, k'$};
\node (E) at (-3.5,1) {$\mathbf{G}$};
\path[->,font=\scriptsize,>=angle 90]
(A) edge node[above]{} (B)
(E) edge node[above]{$\cong$} (A)
(E) edge node[below]{$\phi$} (A)
(C) edge node[below]{$\tau^*$} (D)
(A) edge node[below]{}  (C)
(E) edge node[below]{}  (F)
(F) edge node[below]{Id}  (C)
(B) edge node[below]{}  (D);
\end{tikzpicture}
\end{center}
 \end{rem}

\begin{rem}\label{rem:arithmeticsplitnlattice}
With Remark \ref{rem:galoisaction} in mind, it follows that a choice of $k$ and $n$ determines a commensurability class of an $M_{\Gamma}$ of split-type.   
\end{rem}

\begin{rem}\label{necessary}
By going through the above construction, we identify the following necessary and sufficient additional assumptions to make  $M_{\Gamma}$ quaternion hyperbolic.  
\begin{enumerate}
\item $k$ is totally real (i.e., no $\mathbf{Sp}_{2n}(\C)$ terms)
\item $D$ ramifies at all real places (i.e., no $\mathbf{Sp}_{2n}(\R)$ terms)
\item $h$ is anisotropic at all but one real place, $v_0$ (i.e., only one $\mathbf{Sp}(p_i, n-p_i)$ term)
\item $h\otimes k_{v_0}$ has signature $(n-1,1)$.  (i.e., $G=\mathbf{Sp}(n-1,1)$)
\end{enumerate}
In this case, $M_{\Gamma}$ will be a $4m$-dimensional quaternion hyperbolic orbifold where $m=n-1$.
\end{rem}

\begin{Def}\label{admissibledef}
Let $k$ be a number field with a distinguished infinite place $v_0\in V_k$ and $D\in \mathrm{Br}(k)$ be a quaternion division algebra.
We call the triple $(k,v_0, D)$ \textbf{admissible} if 
$k$ is totally real  and $D$ ramifies at all real places of $k$.
Two admissible triples $(k,v_0,D)$ and $(k',v_0',D')$ are \textbf{equivalent} if there exists a field isomorphism
$\tau:k'\to k$ sending $v'_0$ to $v_0$  and an isomorphism of $k$-algebras $\phi:D\to D\otimes_{k'}k$.
\begin{center}
\begin{tikzpicture}[scale=1.5]
\node (A) at (-1,1) {$D\otimes_{k'} k$};
\node (B) at (1.5,1) {$D'$};
\node (C) at (-1,0) {$k$};
\node (F) at (-3.5,0) {$k$};
\node (D) at (1.5,0) {$k'$};
\node (E) at (-3.5,1) {$D$};
\path[->,font=\scriptsize,>=angle 90]
(B) edge node[above]{} (A)
(E) edge node[above]{$\cong$} (A)
(E) edge node[below]{$\phi$} (A)
(D) edge node[below]{$\tau$} (C)
(C) edge node[below]{}  (A)
(F) edge node[below]{}  (E)
(F) edge node[below]{Id}  (C)
(D) edge node[below]{}  (B);
\end{tikzpicture}
\end{center}
\end{Def}

\qquad
If one wishes to think of $k$ living in $\C$, one may view the distinguished place $v_0$ as the identity embedding.
We will sometimes refer to an \textbf{admissible pair} $(k,D)$ when $v_0$ is clear.

\qquad
It is not hard to see from Remark \ref{necessary} and Construction \ref{arithmeticcnlattice} that given $m\ge2$ and an admissible triple $(k,v_0,D)$, one may build a commensurability class of quaternionic hyperbolic $4m$-orbifolds by considering the $m$-dimensional Hermitian space $(V,h)$ over $D/k$ that has signature $(m+1,0)$ at all but the distinguished real place, where it has signature $(m,1)$.
(Note that by Lemma \ref{uniquelocalHermitian} and \cite[Theorem 7.8.1]{Sch}, $(V,h)$ is uniquely determined at all finite places.)
Hence for each $m\ge2$, this gives a surjective set map 
$$\mathcal{C}_m: \{\mbox{admissible triples}\} \to \{\mbox{commensurability classes of quaternionic hyperbolic $4m$-orbifolds}\}.$$ 


%

\begin{proof}[Proof of Theorem \ref{thrmC}]
In light of the proceeding remarks, all that remains to be shown is that if $\mathcal{C}_m(k,v_0,D)=\mathcal{C}_m(k',v_0',D')$, then $(k,v_0,D)$ and $(k',v_0',D')$ are equivalent.
Let $\mathbf{G}=\mathbf{SU}(V,h)$ and $\mathbf{G}'=\mathbf{SU}(V',h')$.
If $M_h$ and $M_{h'}$ are the same commensurability class,
then by \cite[Proposition 2.5]{PrasadRapinchukWC} and Remark \ref{rem:galoisaction}, there is a field isomorphism $\tau:k'\to k$ and a $k$-rational isomorphism $\phi:\mathbf{G}\to \mathbf{G}'\times_{\mathrm{Spec}\, k'}\mathrm{Spec}\, k$.
The result then follows.
\end{proof}

\section{Constructions of Totally Geodesic Subspaces}
\qquad 
In this section we provide constructions of two important classes of totally geodesic subspaces of a nonsplit-type orbifold $M:=M_\Gamma$, namely \textbf{subform subspaces} and \textbf{(real and complex) restriction subspaces}.  
These should be thought of as generalizations of quaternionic, complex, and real hyperbolic totally geodesic subspaces (\ref{totallydef}) of quaternionic hyperbolic space.
We shall continue to use the notation set forth in Construction \ref{arithmeticcnlattice}.

\begin{con}[Subform Subspaces]\label{subformsubspacedef}
Let $(W,r)$ be a Hermitian $k$-subspace of $(V, h)$.  
Then $\mathbf{H}=\mathbf{SU}(W,r)$ is an absolutely almost simple $k$-subgroup of $\mathbf{G}$. 
Let  $\mathbf{H}':=R_{k/\Q}\mathbf{H}$.  
Then $\mathbf{H'}$ is a semisimple $\Q$-subgroup of $\mathbf{G}'$.  
It follows that $L\cap W$ is an $\mathcal{O}_D$-lattice of $W$, hence $G_L\cap \mathbf{H}'(\R)$ is an arithmetic subgroup of $\mathbf{H}'(\R)$.  
Let $H$ be the image of $\mathbf{H}'(\R)$ under the projection map $\pi$.  
Then $\Lambda:=\pi(G_L\cap \mathbf{H}'(\R))$ is an arithmetic lattice of $H$.  
It follows that $N_{\Lambda}:=\Lambda\backslash H/(H\cap K)$ is commensurable to a totally geodesic subspace of $M_\Gamma$.  
We denote this commensurability class $N_r$.  
In what follows, we shall call such totally geodesic subspaces \textbf{subform subspaces}.  
Observe that for a subform subspace $N\subset M$, $k(N)=k(M)$ and $D(N)=D(M)$.  
Furthermore, if $\dim r\ge 2$ and $r$ is isotropic at a real place of $k$, then $N_r$ is a commensurability class of nontrivial, nonflat, finite volume, locally symmetric spaces of noncompact type.
\end{con}
%

%

\qquad 
Subform subspaces are always of type $C_l$ for $l<n$.  
However, spaces of type $C_n$ have many other types of totally geodesic subspaces.  
In particular, a space of type $C_n$ has totally geodesic subspaces of type $B_l/D_l$ coming from quadratic forms
and outer type ${}^2A_l$ coming from Hermitian forms over a quadratic extension $K/k$.
As we shall now show, these are constructed by restricting the Hermitian form to certain $k$-subspaces of $V$.

\begin{con}[Restriction Subspaces]\label{restrictionsubspacedef}
Fix an orthogonal $D$-basis of $V$ so that $h$ may be represented diagonally by $h=\langle a_1, a_2, \ldots, a_n\rangle$.  
Recall that the Hermitian condition implies that $a_i\in k$.  
Let $W$ denote the $k$-span of this basis and let $(W,q)$ denote the $n$-dimensional quadratic space over $k$ given by $q=\langle a_1, a_2, \ldots, a_n\rangle$.  
Then $(W,q)$ is obtained by restricting $h$ to $W$.
%

\qquad 
The quaternion algebra $D$ contains many quadratic extension of $k$.  
These are called the \textbf{maximal subfields of $D$}, and we denote the collection of isomorphism classes of maximal subfields of $D$ by $\mathrm{Max}(D)$.
If $K\subset D$ is a maximal subfield, then
there exists a pure quaternion $\delta\in D$ such that $K$ is isomorphic to $k(\delta)$.
Let $(W_\delta, h_\delta)$ denote the Hermitian space over $K/k$ given by the restriction of $h$ to $W_\delta:=W\oplus W\delta$.

\qquad
To summarize, we have the following inclusions of absolutely almost simple $k$-groups:
$$\mathbf{SO}(W,q)\subset \mathbf{SU}(W_\delta,h_\delta)\subset \mathbf{SU}(V,h).$$
which translated into inclusions of groups with the following Killing--Cartan types:
$$  \begin{cases}B_{(n-1)/2} &\mbox{ (if $n$ is odd) }\\ D_{n/2} &\mbox{ (if $n$ is even) }\end{cases}\subset {}^2A_{n-1}\subset C_n.$$
%
%

\qquad 
The totally geodesic subspace associated to these subgroups we call \textbf{maximal real (resp. complex) restriction subspaces}.
More generally, the subspaces of $M$ obtained by applying this procedure to a Hermitian subform $r$ of $h$ yields spaces we call \textbf{restriction subspaces}.
The choice of an order $\mathcal{O}_D\subset D$ and $\mathcal{O}_D$-lattice $L\subset D^n$ gives compatible lattices $W$ and $W_\delta$, thereby inducing compatible arithmetic subgroups in $\mathbf{SO}(W,q)$ and $\mathbf{SU}(W_\delta,h_\delta)$, and hence yielding finite volume totally geodesic orbifolds.
Furthermore these spaces have the same field of definition as the ambient space.
\end{con}

\section{Proof of Theorems \ref{mainthm:dimfield} and \ref{mainthm:commensurable}}

\qquad We shall continue the notation of earlier sections, in particular (potentially with subscripts) we let $\mathbf{G}$ denote an absolutely almost simple $k$-group of Killing--Cartan type $C_n$, $G$ its associated semisimple Lie group with no compact factors, and $M$ a finite volume locally symmetric orbifold in the commensurability class determined by $\mathbf{G}$.  
We begin by showing that $\Q TG(M)$ determines some basic structural information about $\mathbf{G}$, namely its absolute rank and whether or not it is $k$-split.

%

\begin{lem}\label{dim}
Let $M_1$ and $M_2$ be finite volume locally symmetric orbifolds of type $C_{n_1}$ and $C_{n_2}$ respectively,  $n_i\ge 3$.  If $\Q TG(M_1)=\Q TG(M_2)$, then
\begin{enumerate}
\item $n_1=n_2$, and
\item either $M_1$ and $M_2$ are both of split-type or both of nonsplit-type.
\end{enumerate}
\end{lem}

\begin{proof}
To begin, suppose $M_1$ is of split-type.  
Then we may take $\mathbf{G}_1=\mathbf{Sp}_{2n_1}(k_1)$ and let $N\subset M$ be the totally geodesic subspace associated to $\mathbf{Sp}_{2n_1-2}(k_1)$.
Since $N\in \Q TG(M_2)$,  $G_2$ contains copies of $\mathbf{Sp}_{2n_1-2}(\R)$ or $\mathbf{Sp}_{2n_1-2}(\C)$, and hence it follows that $n_1\le n_2$.
If $M_2$ is also of split-type, then by symmetry of argument, $n_1=n_2$.  

\qquad We now show that $M_2$ cannot be of nonsplit-type.
We do so by contradiction, namely suppose then that $M_2$ is of nonsplit-type.  
We may take $\mathbf{G}_2=\mathbf{SU}(h_2)$ for some $n_2$-dimensional Hermitian form over a quaternion division algebra $D_2$ over a number field $k_2$.
If $k_2$ has a complex place, or if $D_2$ splits at a real place, (i.e., $q\ne0$ or $s\ne 0$), then we may choose a maximal subform subspace whose isometry group contains $\mathbf{Sp}_{2n_2-2}(\R)$ or $\mathbf{Sp}_{2n_2-2}(\C)$, hence $G_1$ contains $\mathbf{Sp}_{2n_2-2}(\R)$ or $\mathbf{Sp}_{2n_2-2}(\C)$,
from which it follows that $n_1= n_2=:n$.  
Let $\mathbf{H}_2\subset \mathbf{G}_2$ denote the algebraic $k_2$-group associated to a maximal subform subspace of $M$.  
By assumption $R_{k_2/\Q}(\mathbf{H}_2)$ is $\Q$-isogenous to a $\Q$-subgroup $\mathbf{H}'\subset R_{k_1/\Q}(\mathbf{G}_1)$ \cite[Theorem 3.5]{M14spec}.
Writing $k_2\otimes_{\Q}\Q_p =\oplus_{v|p}k_{2,v}$, it is well known \cite[2.1.2]{PlRa} that $R_{k_2/\Q}(\mathbf{H}_2)(\Q_p)\cong \prod_{v|p}\mathbf{H}_2(k_{2,v})$.
In particular, for primes over which there is a place $v\in V_{k_2}$ where $D_2$ ramifies, this is a product of $k_{2,v}$-groups of type $C_{n-1}$, some of which are not $k_{2,v}$-split.
However, all $\Q_p$-subgroup $\mathbf{H}'(\Q)$ are a product of split groups over $k_{1,v}$ of type $C_{n-1}$.
Hence these two groups cannot have the same localizations, and therefore these could not have been $\Q$-isogenous.
We have found a subform subspace of $M_2$ not commensurable to a subspace of $M_1$, thereby producing a contradiction.  
%
We now have only to check the case when $D_2$ ramifies at every real place (i.e., $q=0$ and $s=0$).
Then $G=\prod_1^{r_2} \mathbf{Sp}(p_i,n_2-p_i)$, and we may choose a maximal subform subspace which is also of this form.
Then the group $\prod_1^{r_2} \mathbf{Sp}(p'_i,n_2-1-p_i')$, where $p_i'$ is either $p_i$ or $p_1-1$, sits as a Lie subgroup of $(\mathbf{Sp}_{n_1}(\R))^{q_1}\times (\mathbf{Sp}_{n_1}(\C))^{s_1}$, where $n_1\le n_2$, but this cannot happen.
We have therefore shown (b).

\qquad We have left to prove (a) when both spaces are of nonsplit-type.  
We may now take $\mathbf{G}_i:=\mathbf{SU}(h_i)$, $i=1,2$, where  $h_i$ are $n_i$-dimensional Hermitian forms over quaternion algebras $D_i$ over number fields $k_i$.  
We shall prove the contrapositive.  
If $n_1\ne n_2$, then (potentially after relabeling), $\dim h_1>\dim h_2$.  
Let $v_0\in V_{k_1}$ be an infinite place where $h_{1,v_0}:=h_1\otimes_{k_1} k_{1, v_0}$ is isotropic.  
By deleting one entry in a diagonal representation of $h_1$, we have an $(n_1-1)$-dimensional Hermitian form, which we denote $r$, over $D_1$ which is isotropic over $k_{1,v_0}$.
Let $\mathbf{H}:=\mathbf{SU}(r)$, $H$ denote the projection of $R_{k_1/\Q}(\mathbf{H})(\R)$ via $\pi$, and $N$ denote the corresponding subform subspace.
Considering the dimensions of the simple factors of $H$,  it follows that $H$ cannot be isogenous to a proper subgroup of $G_2$.
By construction $N$ a nonflat finite volume totally geodesic subspace of $M_1$ which $N$ cannot be commensurable a proper totally geodesic subspace of $M_2$.   
The result then follows.
\end{proof}

\qquad When we restrict to the cases where $G$ is simple, $\Q TG(M)$ determines additional structural information of $\mathbf{G}$, namely its $\R$-rank.

\begin{lem}\label{rank}
Let $M_1$ and $M_2$ be simple  finite volume locally symmetric orbifolds of type $C_{n_1}$ and $C_{n_2}$ respectively,  $n_i\ge 2$.  If $\Q TG(M_1)=\Q TG(M_2)$, then
$\mathrm{rank} (M_1)=\mathrm{rank} (M_2)$
\end{lem}

\begin{proof}
By Lemma \ref{dim}, $n_1=n_2$.  
Let $v_i\in V_{k_i}$ be the unique real place where $\mathbf{G}_i$ is isotropic, it suffices to show that the $\R$-rank of $\mathbf{G}_1\otimes k_{1,v_1}$ equals the $\R$-rank of $\mathbf{G}_2\otimes k_{2,v_2}$.  
We will prove the contrapositive.  Suppose $\mathrm{rank}_\R(\mathbf{G}_1\otimes k_{1,v_1})>\mathrm{rank}_\R(\mathbf{G}_2\otimes k_{2,v_2})$.  Recall that $\mathrm{rank}_\R(\mathbf{Sp}_{2n}(\R))=n$ and $\mathrm{rank}_\R(\mathbf{Sp}(p, n-p))=\min\{p, n-p\}$.
Let $r$ be the Hermitian subform as in the proof of Lemma \ref{dim}.  
We can choose $r$ to guarantee that $\mathrm{rank}_\R(\mathbf{SU}(r\otimes k_{v_1}))\ge \mathrm{rank}_\R(\mathbf{G}_2\otimes k_{2,v_2})$, and hence $N_r\notin \Q TG(M_2)$.
\end{proof}

%

%

\begin{proof}[Proof of Theorem \ref{mainthm:dimfield}]
(a) 
Begin by noting that $\overline{G}_i$ is isomorphic (as a Lie group) to the adjoint group of $G_i$.
Lemma \ref{dim} and Lemma \ref{rank} imply $\overline{G}_1$ and $\overline{G}_2$ have the same absolute rank $n$ and same $\R$-rank.
Simple algebraic groups of type $C_n$ are determined up to isogeny by their rank and $\R$-rank \cite[Lemma 2.2]{LinowitzMeyer}, and hence $\overline{G}_1$ and $\overline{G}_2$ are isomorphic.
From this it immediately follows that $\dim M_1=\dim M_2$.
%

(b)  Now when $n\ge 4$, 
$\dim \mathbf{G}_i=n(2n+1)<2(n-1)(2n-1)=2\dim \mathbf{H}.$
Since $M_1$ and $M_2$ are simple, by \cite[Proposition 7.4]{M14spec} the result follows.
\end{proof}

\begin{proof}[Proof of Theorem \ref{mainthm:commensurable}]
By Lemma \ref{dim} we may assume $n_1=n_2$ and that either they are both of split-type or both of nonsplit-type.
Let $k$ be denote a fixed representative of the isomorphism class of $k(M_1)$ and $k(M_2)$.
Suppose first that both $M_1$ and $M_2$ are of split-type.
By Remark \ref{rem:arithmeticsplitnlattice}, it follows that a choice of absolute rank and of field of definition determines their commensurability class and the result follows.

\qquad
Now suppose both spaces are of nonsplit-type.
Then there are $n$-dimensional Hermitian forms $h_i$ over quaternion algebras $D_i$ over $k$ giving rise to $M_i$.
Let $\mathbf{H}_1=\mathbf{SU}(r)$ be associated to a maximal subform subspace of $M_1$.
Observe that since $R_{k/\Q}(\mathbf{H}_1)(\Q_p)\cong \prod_{v|p} \mathbf{H}_1(k_v)$ for each rational prime $p$, $R_{k/\Q}(\mathbf{H}_1)$ detects the local ramification type of $D_1$ over each $p$.
By assumption and \cite[Theorem 3.5]{M14spec}, $R_{k/\Q}(\mathbf{H}_1)$ is $\Q$-isogenous to a $\Q$-subgroup of $R_{k/\Q}\mathbf{G}$,
and hence $D_2$ must have the same ramification behavior over each rational prime.
Therefore, this isogeny gives a field automorphism $\tau:k\to k$ sending the ramification set of $D_1$ to the ramification set of $D_2$.
Upon twisting by $\tau$, we may take $D_1$ and $D_2$ to be $k$-isomorphic algebras, and we will let $D$ denote a fixed representative of this isomorphism class.

\qquad 
Suppose now that $M_1$ and $M_2$ are noncommensuable.
Noncommensuable spaces with the same field and algebra of definition must come from nonisometric Hermitian forms, which in turn have nonisometric trace forms \cite[Theorem 10.1.7]{Sch}.  
Since  the trace form is uniquely determined at every finite place (see Lemma \ref{localforms} and \cite[Theorem 7.8.1]{Sch}),  every complex place, and every real place where $D$ splits, they must differ at a real place where $D$ ramifies.
Furthermore, by \cite[Lemma 8.2]{M14spec}, we may replace $h_i$ with a similar form such that the signature $(s_1,s_2)$ of $h_i$ at each real place where $D$ ramifies satisfies $s_1\ge s_2$.
Note that if $h_1$ and $h_2$ are isotropic at a different number of infinite places, we can detect this with a subform subspace that lies in one of the $M_i$'s but not in the other.
Suppose for the moment $n\ge 5$.
By the same signature arguments used in \cite[Theorem 8.8]{M14spec}, it follows that there is a maximal subform subspace of one whose signatures force it to not be a subspace of the other.
In the case $n=3$, since isotropic places must have signature $(2,1)$, to be noncommensurable, they must have a different number of isotropic place, which can be detected by a maximal subform subspace.
In the case $n=4$, either they have a different number of isotropic places, which can be detected by a subform subspace, or they must have a different number of places of signature $(2,2)$ and $(3,1)$ and hence we may detect this with a maximal complex restriction subspace.
In each of these cases, we have constructed a totally geodesic subspace contained in one but not the other, thereby producing a contradiction.
\end{proof}

\qquad The following corollary is an immediate consequence of Theorem \ref{mainthm:dimfield} and Theorem \ref{mainthm:commensurable}.

\begin{cor}\label{cor:bandc}
Let $M_1$ and $M_2$ be simple finite volume locally symmetric orbifolds of type $C_{n_1}$ and $C_{n_2}$, respectively, where $n_i\ge 4$.
If $\Q TG(M_1)=\Q TG(M_2)$, then $M_1\sim_cM_2$.
\end{cor}

\section{Applications to Quaternionic Hyperbolic Spaces}

\qquad 
In this section we apply our techniques to answer some questions about quaternionic hyperbolic orbifolds.  
Throughout this section, let $M$ be a finite volume, $4m$-dimensional, quaternionic hyperbolic orbifold, $m=n-1\ge 2$, with commensurability class determined by the \hyperref[admissibledef]{admissible triple} $(k,v_0,D)$ (see Theorem \ref{thrmC}).
We fix once and for all the real place $v_0\in V_k$ over which the $n$-dimensional Hermitian form $h$ is isotropic.
We denote $k_{v_0}$ by $\R$ and $D_{k_{v_0}}$ by $\mathbb{H}$.
Let $\mathbf{G}=\mathbf{SU}(V,h)$,  $V_\R=V\otimes _k\R$, $h_\R$ denote the extension of $h$ to $V_\R$, and $\mathbf{G}(\R)=\mathbf{Sp}(m,1)$.
We let $\mathrm{Max}(D)$ denote the set of isomorphism classes of maximal subfields of $D$.

\qquad 
By Theorem \ref{totgeoclassification},  we know $M$ has three types of totally geodesic subspaces: real, complex, and quaternionic hyperbolic.  
We show in Proposition \ref{fieldofdefquathyp} that the finite volume totally geodesic subspaces of $M$ inherit the arithmetic structure of $M$.
In particular, the real hyperbolic are standard arithmetic and the complex hyperbolic are arithmetic of the first kind.
One can find the construction of standard arithmetic real hyperbolic orbifolds (which are sometimes referred to as arithmetic orbifolds of simplest type) in \cite[\S2]{Millson}
and the construction of arithmetic complex hyperbolic orbifolds of the first kind in \cite[Chp 5]{McRey2}.



\begin{prop}[Classification of Finite Volume Subspaces]\label{fieldofdefquathyp}
If $N\in \Q TG(M)$, then $N$ is arithmetic with field of definition $k(N)=k$ and $N$ can be realized by restricting $h$ to some $k$-subspace of $V$. In particular, if
\begin{enumerate}
\item $N$ is real hyperbolic, then $N$ is standard arithmetic.
\item $N$ is complex hyperbolic, then $N$ is arithmetic of the first kind relative to $K/k$ where $K\in \mathrm{Max}(D)$.
\item $N$ is quaternionic hyperbolic, then $D(N)=D(M)$.
\end{enumerate}
\end{prop}

\begin{proof}
By  \cite[Theorem 3.3]{M14spec}, $N$ is arithmetic.  
Let $\mathbb{F}$ denote $\R$, $\C$, or $\mathbb{H}$ according the whether $N$ is real, complex, or quaternionic.
Let $H\subset G:=\mathbf{G}(\R)$ be the connected semisimple Lie subgroup giving rise to $N$.  
Since $N$ is $\mathbb{F}$-hyperbolic, it follows that $H=\mathbf{H}(\R)^\circ$ where $\mathbf{H}=\mathbf{SU}(W',r')$ for some totally $\mathbb{F}$-subspace $W'\subset V_\R$ and $r'$ the restriction of $h_\R$ to $W'$. 
Let $L\subset V$ be an $\mathcal{O}_D$-lattice and let $G_{L}$ be its stabilizer in $G$.  
Since $\Lambda:=G_{L}\cap H$ is a lattice in $H$, by Borel's density theorem \cite{B60}, $\Lambda$ is Zariski-dense in $\mathbf{H}$, and hence $\mathbf{H}$ is defined over $k$ \cite[AG.14.4]{B1}.
From this it follows that $k(N)=k$.

\qquad 
Furthermore, $W:=W'\cap V$ is a $k$-structure on $W'$ and letting $r$ denote the restriction of $h$ to $W$,
we have that the commensurability class of $N$ is determined by the $k$-group $\mathbf{SU}(W,r)$.  
In the case that $W'$ is totally real, $(W',r')$ is a quadratic space over $\R$, and hence $(W,r)$ is a quadratic space over $k$ and (1) follows.
In the case that $W'$ is totally complex, let $\delta\in \mathbb{H}$ denote the pure quaternion for which $W'=W'\delta$.
It follows that the $k$-algebra $K:=k[h(W,W)]\subset (D\cap \R(\delta))$ is a maximal subfield of $D$ and hence is an imaginary quadratic extension of $k$.
It follows that $(W,r)$ is a Hermitian space over the pair $K/k$ and (2) follows.
In the case that $W$ is totally quaternionic, the result follows from the fact $N$ is a subform subspace.
\end{proof}

\qquad
We can now analyze each of these classes of totally geodesic subspaces on their own. 
In particular, it is natural to the ask the following question:
\begin{ques}
To what extent does the collection of finite volume, real (resp. complex) hyperbolic, totally geodesic subspaces of $M$ determine the commensurability class of $M$?
\end{ques}

\qquad 
Interestingly we shall show that the answer is a geometric version of the following fact about quaternion division algebras over a number field $k$: 
there are many quaternion division algebras with center $k$, but a quaternion division algebra is uniquely determined by its collection of maximal subfields \cite[Remark 5.4]{PrasadRapinchukWC}.  
Here we should think of the center of the division algebra as the real hyperbolic subspaces, and the maximal subfields as the complex hyperbolic subspaces.


\begin{prop}\label{parametrizecomplexhyp}
Let $N$ be an arithmetic $2m$-dimensional complex hyperbolic orbifold of the first kind associated to the pair $K/k$, where $K\in \mathrm{Max}(D)$.  
Then $N$ is commensurable to a totally geodesic subspace of $M$.
This subspace is realized by restricting the Hermitian form $h$ on $V$ to a $2n$-dimensional totally complex subspace.
\end{prop}

\begin{proof}
Associated to $N$ is an $n$-dimensional Hermitian space, $(W^*,h^*)$, over $K/k$ with signature $(m,1)$ at the real place where $h$ is isotropic and anisotropic at all other real places.
By assumption, there exists a pure quaternion $\delta \in D$ such that $K$ is isomorphic to $k(\delta)$.
Let $\mu$ be a pure quaternion compliment to $\delta$ (i.e., $D$ is generated as a $k$-algebra by $\delta$ and $\mu$ and $\delta\mu=-\mu\delta$).
In particular, $D=K\oplus K\mu$ and we can naturally extend $W^*$ to a $D$-vector space $W':=W^*\oplus W^*\mu$.
Furthermore, we can extend $(W^*,h^*)$ to a Hermitian space $(W',h')$ over $D$ via the following:
For $a+b\mu,c+d\mu\in D$, $a,b,c,d\in k(\delta)$, $$h'(a+b\mu, c+d\mu):= (h^*(a,c)-h^*(b,d)\mu^2) + (h^*(a,d)-\overline{h^*(b,c)})\mu.$$
By construction, $h$ and $h'$ are locally isometric at the infinite places of $k$.
It follows from Lemma \ref{localforms} that the trace forms $q_h$ and $q_{h'}$ are locally isometric at all finite places of $k$.
By local-to-global uniqueness \cite[Theorem 66:4]{OM}, the trace forms are isometric over $k$, and hence by \cite[Theorem 10.1.7]{Sch}, the Hermitian spaces $(V,h)$ and $(W',h')$ are isometric over $k$.
Furthermore, we recover $h^*$ by restricting $h'$ to the totally complex subspace $W^*\subset W'$, and the result follows.
\end{proof}

\qquad Recall that $\Q TG(M, \mathcal{P}_{\mathbb{C}})$ (resp. $\Q TG(M, \mathcal{P}_{\mathbb{H}})$) is the collection of commensurability classes of maximal, finite volume, nonflat, complex (resp. quaternionic) hyperbolic, totally geodesic subspace of $M$.

\begin{proof}[Proof of Theorem \ref{mainthm:complex}]
First suppose  $\Q TG(M_1, \mathcal{P}_{\mathbb{C}})=\Q TG(M_2, \mathcal{P}_{\mathbb{C}})$.
It is clear that $\dim M_1=\dim M_2$, and by Proposition \ref{fieldofdefquathyp},
$k(M_1)$ and $k(M_2)$ are isomorphic.  
Let $k$ be a fixed representative of this class, $v_0$ be its distinguished real embedding, and identify $D(M_1)$ and $D(M_2)$ as algebras over $k$.
By Proposition \ref{parametrizecomplexhyp}, it follows that $\mathrm{Max}(D(M_1))=\mathrm{Max}(D(M_2))$. 
Since the set of maximal subfields determines the isomorphism class of the algebra \cite[Remark 5.4]{PrasadRapinchukWC}\cite{M14div}, we conclude that $D(M_1)$ and $D(M_2)$ are $k$-isomorphic. 
It follows that $(k,v_0,D(M_1))$ and $(k,v_0,D(M_2))$ are equivalent admissible triples.
By Theorem \ref{thrmC},  $M_1$ and $M_2$ are commensurable.

\qquad 
Now suppose $\Q TG(M_1, \mathcal{P}_{\mathbb{H}})=\Q TG(M_2, \mathcal{P}_{\mathbb{H}})$.
Again it is clear that $\dim M_1=\dim M_2$.   
By Proposition \ref{fieldofdefquathyp}, $k(M_1)$ and $k(M_2)$ are isomorphic.
Let $k$ be a fixed representative of this class, $v_0$ be its distinguished real embedding.
By Proposition \ref{fieldofdefquathyp}\textit{(3)},  $D(M_1)$ and $D(M_2)$ are isomorphic.
It follows that $(k,v_0,D(M_1))$ and $(k,v_0,D(M_2))$ are equivalent admissible triples, 
and again by Theorem \ref{thrmC},  $M_1$ and $M_2$ are commensurable.
\end{proof}

\begin{prop}\label{parametrizerealhyp}
Let $N$ be a standard arithmetic $m$-dimensional real hyperbolic orbifold with field of definition $k$.  
Then $N$ is commensurable to a totally geodesic subspace of $M$.
This subspace is realized by restricting the Hermitian form on $V$ to an $n$-dimensional totally real subspace.
\end{prop}

\begin{proof}
Associated to $N$ is an $n$-dimensional quadratic space $(W^*,q)$ over $k$ with signature $(m,1)$ at $v_0$ and has signature $(m+1,0)$ at all other real places.
Let $W'=W^*\otimes_k D$ and $h':=q\otimes_k D$  be the Hermitian form over $D$ extending $q$ to $W'$.  
By Lemma \ref{localforms} the trace forms $q_h$ and $q_{h'}$ are locally isometric.
As in the proof of Proposition \ref{parametrizecomplexhyp}, \cite[VI.66:4]{OM}, and \cite[Theorem 10.1.7]{Sch} imply the Hermitian spaces $(V,h)$ and $(W',h')$ are isometric over $k$.
Furthermore, we recover $q$ by restricting $h'$ to the totally real subspace $W^*\subset W'$, and the result follows.
\end{proof}

\begin{cor}\label{parametrizerealhyp2}
Let $S$ be a standard arithmetic, $d$-dimensional, real hyperbolic orbifold with field of definition $k$, $2\le d\le m$.  
Then $S$ is commensurable to a totally geodesic subspace of $M$.  
This subspace is realized by restricting the Hermitian form on $V$ to an $d$-dimensional totally real subspace.
\end{cor}

\begin{proof}
Let $N$ be a standard arithmetic, $m$-dimensional, real hyperbolic orbifold with field of definition $k$ containing $S$ as a totally geodesic subspace.
By Proposition \ref{parametrizerealhyp}, $N$ is commensurable to a totally geodesic subspace of $M$ and the result follows.
\end{proof}

%

\qquad 
Proposition \ref{parametrizerealhyp} says that $\Q TG(M, \mathcal{P}_{\mathbb{R}})$ cannot distinguish between commensurability classes of quaternionic hyperbolic orbifolds with a fixed field of definition $k$.
In the following corollary, we give a sense of the extent to which $\Q TG(M, \mathcal{P}_{\mathbb{R}})$ fails to determine a commensurability class.

\begin{cor}\label{cor:oneintoall}
For each $m\ge 2$ and $s\ge 2$, there exists a family $\{M_1,M_2, \ldots, M_s\}$ of $4m$-dimensional quaternionic hyperbolic orbifolds such that the members are pairwise noncommensurable and every finite volume, real hyperbolic, totally geodesic subspace in one is commensurable to a totally geodesic subspace in each of the others.
Furthermore, this collection can be chosen so that each member has the same field of definition.
\end{cor}

\begin{proof}
Fix some totally real number field $k$, a distinguished real place $v_0$, and pick $s$ quaternion division algebras $D_1, D_2,\ldots D_s\in \mathrm{Br}(k)$ that are pairwise nonisomorphic
 and which ramify at all real places of $k$.  
Let $M_i$ be an $4m$-orbifold in the commensurability class determined by the triple $(k,v_0,D_i)$.
Since the collection of admissible triples $(k,v_0,D_i)$ are pairwise inequivalent, $\{M_1,M_2, \ldots, M_s\}$ are pairwise noncommensurable (Theorem \ref{thrmC}).
By Proposition \ref{fieldofdefquathyp}, the real hyperbolic, totally geodesic subspaces of the $M_i$ are standard arithmetic with field of definition $k$.
Let $S$ be a standard arithmetic, $d$-dimensional, real hyperbolic orbifold with field of definition $k$.
By Corollary \ref{parametrizerealhyp2}, $S$ is commensurable to a totally geodesic subspace of $M_i$ for all $1\le i\le s$.
\end{proof}

The proof of Theorem \ref{mainthm:real} immediately follows from this corollary.

%
%
%
%


\section{Surface Subgroups and Theorem \ref{mainthm:surfacesubgroup}}

\qquad 
We now use our results on real hyperbolic totally geodesic subspaces to construct (infinitely many commensurability classes of) quasiconvex surface subgroups within the fundamental group of a quaternionic hyperbolic orbifold, thereby proving Theorem \ref{mainthm:surfacesubgroup}.
In this section, $M$ will denote a $4m$-dimensional compact quaternionic hyperbolic orbifold, $m\ge 2$, whose commensurability class is given by the admissible triple $(k,v_0,D)$.



\begin{lem}\label{compactcondition}
Let $M$ be a simple arithmetic locally symmetric space of noncompact type.  If $k(M)\ne \Q$,  then $M$ is compact.
\end{lem}

\begin{proof}
By our assumption that $M$ is arithmetic and simple, it comes from the restriction of scalars from $k$ to $\Q$ of a simple algebraic $k$-group $\mathbf{G}$ where $k$ is some totally real number field not equal to $\Q$.
Since $M$ is simple, $\mathbf{G}$ can only be isotropic at one real place, and hence $\mathbf{G}$ must be $k$-anisotropic
By \cite[11.8]{BoHC}, $M$ is compact if and only if $R_{k/\Q}\mathbf{G}$ is $\Q$-anisotropic  which is the case if and only if $\mathbf{G}$ is $k$-anisotropic, and the result follows.
\end{proof}

\begin{lem}\label{compactcondition2}
Let $M$ be a quaternionic hyperbolic $4m$-orbifold, $m\ge 2$.  If $M$ is compact, then $k(M)\ne \Q$.
\end{lem}

\begin{proof}
We prove the contrapositive.
Suppose $k(M)=\Q$.
By assumption, $M$ is associated to an $(m+1)$-dimensional Hermitian form over an admissible pair $(\Q, D)$.
Since the trace form $q_h$ has dimension greater than 4, by \cite[Lemma 4.2.7]{Cassels}, it is isotropic at all finite places.  
By assumption $h$, and hence $q_h$, is isotropic at the real place.
Hence $q_h$ is isotropic at all places, and by the Strong Hasse Principle \cite[Theorem 6.1.1]{Cassels}, $q_h$ is isotropic, which implies $h$ is isotropic.
It follows that $\mathbf{SU}(h)$ is $\Q$-isotropic, and hence by  \cite[11.8]{BoHC}, $M$ is noncompact.
\end{proof}

\qquad 
Combining these two results, we obtain the following proposition.

\begin{prop}\label{realhyperbolicsurface}
If $M$ is a compact quaternionic hyperbolic $4m$-orbifold, $m\ge 2$, then every arithmetic real hyperbolic surface with field of definition $k=k(M)$ is commensurable to a closed, orientable, real hyperbolic, totally geodesic surface in $M$.
\end{prop}

\begin{proof}
Let $S'$ be an arithmetic hyperbolic surface with field of definition $k$.
Since all arithmetic real hyperbolic surfaces are standard arithmetic, Corollary \ref{parametrizerealhyp2} implies $S'$ is commensurable to a real hyperbolic, totally geodesic surface in $M$, which we denote $S$.
By Lemma \ref{compactcondition2} and the fact $M$ is compact, we know $k\neq \Q$, and by 
Lemma \ref{compactcondition} it follows that $S$ is compact.
The metric on $S$ is complete, hence it has no boundary and is closed.
Now let $\widetilde{S}$ be a simply connected cover of $S$ in $\mathbf{H}_{\mathbb{H}}^m$.  
Since the determinant of any matrix in $\mathbf{Sp}(m,1)\subset \mathbf{Sp}_{2m+2}(\C)$ is 1, the stabilizer of $\widetilde{S}$ in $\mathbf{PSp}(m,1)$ is $\mathbf{PSO}(2,1)$ (i.e., all isometries of $\mathbf{H}_{\mathbb{H}}^m$ stabilizing $\widetilde{S}$ are orientation preserving) and therefore $S$ is orientable.
\end{proof}

\begin{proof}[Proof of Theorem \ref{mainthm:surfacesubgroup}]
Let $S\subset M$ be one of the closed, orientable, totally geodesic, real hyperbolic surfaces produced in Proposition \ref{realhyperbolicsurface}.
By construction, $\pi_1(S)$ sits as a subgroup of $\pi_1(M)$.
%
%
%
%
%
Since $M$ is compact, the \v{S}varc--Milnor Lemma \cite[9.19]{Bridson} gives that $\pi_1(M)$ and $M$ are quasi-isometric.
Since $S$ is totally geodesic in $M$, it is convex, and hence the above quasi-isometry gives that
$\pi_1(S)$ is quasiconvex in $\pi_1(M)$.
%
\end{proof}

\begin{thm}\label{thm:relhyp}
If $\Gamma<\mathbf{Sp}(m,1)$, $m\ge2$, is a nonuniform lattice, then $\Gamma$ contains surface subgroups.
\end{thm}

\begin{proof} 
It suffices to show that, for $m\ge 2$, every noncompact, finite volume, quaternionic hyperbolic $4m$-orbifold  $M$ contains a compact, real hyperbolic, totally geodesic surface.  
By Lemma \ref{compactcondition2}, it follows that $k(M)=\Q$.
Let $(\Q, D)$ be the admissible pair associated to $M$.
Let $S'$ be a compact surface with field of definition $\Q$.  
Corollary \ref{parametrizerealhyp2} implies $S'$ is commensurable to a real hyperbolic, totally geodesic surface $S$ in $M$.
Then $\pi_1(S)$ is a surface subgroup in $\Gamma=\pi_1(M)$, and the result follows.
\end{proof}

\qquad 
Since nonuniform lattices in $\mathbf{Sp}(m,1)$ are relatively hyperbolic, 
Theorem \ref{thm:relhyp} gives the existence of a class of relatively hyperbolic groups that contain surface subgroups.


\end{document}